\RequirePackage{fix-cm}
\documentclass[smallextended]{svjour3}       
\smartqed  
\usepackage{graphicx}

\usepackage{amsmath}
\usepackage{amsfonts}
\usepackage{amssymb}
\usepackage{amscd}

\usepackage{hyperref}
\begin{document}

\title{On $1$-absorbing prime submodules}


\author{Emel Aslankarayigit Ugurlu        
}


\institute{E. Aslankarayigit Ugurlu \at
             Department of Mathematics, Marmara University, Istanbul, Turkey. \\ \email{emel.aslankarayigit@marmara.edu.tr}           
     }

\date{Received: 02 July 2020 / Accepted: date}

\maketitle

\begin{abstract}
In this study, we aim to introduce the concept of \textit{1-absorbing prime
	submodule} of an unital module over a commutative ring with non-zero
identity. Let $M\ $be an $R$-module and $N$ be a proper submodule of $M$.
For all non-unit elements $a,b\in R$ and $m\in M$ if $abm\in N,$ either $%
ab\in(N:M)$ or $m\in N,$ then $N$ is called 1-absorbing prime submodule of $%
M.$ We show that the new concept is a generalization of prime submodules at
the same time it is a kind of special 2-absorbing submodule. In addition to
some properties of a 1-absorbing prime submodule, we obtain a
characterization of it in a multiplication module.

\keywords{1-absorbing prime ideal \and 1-absorbing prime submodule \and prime submodule}
\subclass{16D10 \and 16D80  }
\end{abstract}

\section{Introduction}

\label{Sec:1}

In this article, we focus only on commutative rings with non-zero identity
and non-zero unital left modules. Let $R$ always denote such a ring and let $M\ $%
denote such an $R$-module. The concept of prime ideals and its
generalizations have a significant place in commutative algebra since they
are used in understanding the structure of rings. Recall that a proper ideal 
$P$ of $R$ is said to be \textit{prime ideal} if $ab\in P$ implies $a\in P$
or $b\in P$, \cite{atiyah}. Several authors have extended the notion of
prime ideals to modules, see, for example, \cite{JD1978,L,MM}. A proper
submodule $N$ of a module $M$ over a commutative ring $R$ is said to be 
\textit{prime} \textit{submodule} if whenever $rm\in N$ for some $r\in
R,m\in M$, then $m\in N$ or $rM\subseteq N,$ \cite{RA2003}.

$Id(R)$ and $S(M)$ denote the lattice of all ideals of $R$ and the lattice
of all submodules of $M,$ respectively. The \textit{radical of} $I,$ denoted
by $\sqrt{I},$ is defined as the intersection of all prime ideals contain $I.
$ Note that we have the equality $\sqrt{I}=\{r\in R$ $|$ $r^{k}\in I$ for
some $k\in%
\mathbb{N}
\},$ see \cite{atiyah}. For any $a\in R,$ the principal ideal generated by $a
$ is denoted by $(a).$ All unit elements of $R$
is denoted by $U(R).$ For any element $x\in M,$ the set \ $%
<m>=Rm=\{rm:\forall r\in R\}$ is the cyclic submodule of $M.$ If $M=<X>,$
we say that $M$ is a finitely generated $R$-module for any finite subset $X$
of $M.$ Now we define the residue of $N$ by $M$. If $N$ is a
submodule of an $R$-module $M,$ the ideal $\{r\in R:rM\subseteq N\}$ is
called the \textit{residue} of $N$ by $M$ and it is denoted by $(N:_{R}M).$
If $R$ is clear, it is written by only $(N:M).$ In particular, $(0_{M}:_{R}M)
$ is called the \textit{annihilator of} $M$ and denoted by $Ann(M),$ see 
\cite{smith}. If the annihilator of $M$ equal to $0_{R},$ then $M$ is called a 
\textit{faithful module}. For an element $m\in M,$ the \textit{%
	annihilator of} $m$ is defined as $Ann(m):=\{r\in R:rM\subseteq N\}$ and it
is an ideal of $R.$ For a proper submodule $N$ of $M$, the \textit{radical of} $N,$ denoted by $rad(N),$ is defined to be the intersection of all prime
submodules of $M$ containing $N.$ If there is no prime submodule containing $%
N$, then $rad(N)=M,$ see \cite{smith}.

In 2007, the notion of 2-absorbing ideal, which is a generalization of prime
ideal, was introduced by Badawi as the following: a proper ideal $I$ of $R$
is called a \textit{2-absorbing ideal} of $R$ if whenever $a,b,c\in R$ and $%
abc\in I$, then $ab\in I$ or $ac\in I$ or $bc\in I$, see \cite{Ba}. Then in
2011, A. Y. Darani and \ F. Soheilnia defined the concept of \ 2-absorbing
submodule as following: $N$ is said to be a \textit{2-absorbing submodule}
of $M$ if whenever $a,b\in R$ and $m\in M$ with $abm\in N$ then $ab\in(N:M)$
or $am\in N$ or $bm\in N,$ see \cite{darani}. Actually, the concept of 2-absorbing submodule is a generalization of prime submodules.

Recently, Yassine et al. defined a new class of ideals, which is an
intermediate class of ideals between prime ideals and 2-absorbing ideals. A
proper ideal $I$ of $R$ is said to be a \textit{1-absorbing prime ideal} if
for each non-units $a,b,c\in R$ with $abc\in I$, then either $ab\in I$ or $%
c\in I$, see \cite{yassine}. Note that every prime ideal is a 1-absorbing
prime and every 1-absorbing prime ideal is a 2-absorbing ideal. Thus we have
a chain: prime ideals $\Rightarrow$ 1-absorbing ideals $\Rightarrow$ 2
absorbing ideals. On the other hand, we have a second chain: prime
submodules $\Rightarrow$ 2 absorbing submodules. Thus we realize that there
is a missing part in the second chain, which is between prime submodules and
2-absorbing submodules. Then we define the missing part of the chain as 
\textit{1-absorbing prime submodules.}

In this paper, after introducing the notion of \textit{1-absorbing prime submodule} of an unital left module over a commutative ring with non-zero identity, we examine the properties of the new class. For all non-unit elements $a,b\in R$ and $m\in M,$ if $%
abm\in N,$ either $ab\in(N:M)$ or $m\in N,$ then $N$ is called \textit{%
	1-absorbing prime\textbf{\ }submodule} of $M,$ see Definition \ref{def 1}.
Firstly, we investigate in Proposition \ref{mainpro}, the relation between
1-absorbing prime submodules and other classical submodules such as prime
submodules, 2-absorbing submodules. We prove that every prime submodule is a
1-absorbing submodule, but the converse is not true: To see this, consider
the cyclic submodule of $%
\mathbb{Z}
_{4}-$module $%
\mathbb{Z}
_{4}[X]$ generated by $X,$ that is, $<X>$. Indeed, it is 1-absorbing
submodule, but is not a prime submodule of $%
\mathbb{Z}
_{4}-$module $%
\mathbb{Z}
_{4}[X],$ see Example \ref{ex2}.  Furthermore, we show that every 1-absorbing prime submodule is a
2-absorbing submodule. However it is not true that every 2-absorbing submodule
is a 1-absorbing prime submodule. For instance, consider the cyclic
submodule of $%
\mathbb{Z}
-$module $%
\mathbb{Z}
_{30}$ generated by $\overline{6},$ it is a 2-absorbing submodule but not
a 1-absorbing prime submodule, see Example \ref{ex3}. Actually, by the help of Proposition \ref{mainpro}, the second chain is completed. For the completed picture of these algebraic structures, see Figure 1. Among other
results in this paper, in Section 2, we proved that $N$ is a 1-absorbing
prime submodule of $M$ $\Leftrightarrow$ for any two proper ideals $I,J$ and
for a proper submodule $K$ of $M,$ $IJK\subseteq N$ implies either $%
IJ\subseteq(N:M)$ or $K\subseteq N$, see Theorem \ref{theoremN}. In Corollary \ref{corQ}, after Theorem \ref{theoremQ},  we conclude that if $R$ is a quasi local ring the concepts of 1-absorbing prime submodule and prime submodule are equivalent. After we
investigate the behavior of 1-absorbing prime submodules under
homomorphisms, in Corollary \ref{corbolum}, we prove that if $N$ is a
1-absorbing prime submodule of $M$, then $N/K$ is a 1-absorbing prime
submodule of $M/K$. In Section 3, we examine 1-absorbing prime submodules of
a multiplication module. Under special conditions, we show that for a 1-absorbing prime ideal $I$ of $R,$ $abm\in
IM$ implies $ab\in I$ or $m\in IM$ for all
non-units $a,b\in R$ and $m\in M,$ see Theorem \ref{main}. By using this
result in Theorem \ref{char}, we obtain a characterization of 1-absorbing prime submodules
in multiplication modules.


\section{Properties of 1-Absorbing Prime Submodules of Modules}

\begin{definition}
	\label{def 1}Let $M\ $be an $R$-module and $N$ be a proper submodule of $M$.
	For all non-units element $a,b\in R$ and $m\in M,$ if $abm\in N,$ either
	$ab\in(N:M)$ or $m\in N,$ then $N$ is called \textbf{1-absorbing prime
		submodule} of $M.$
\end{definition}

It is clear that if $I$ is a 1-absorbing prime ideal of $R,$ it is a
1-absorbing prime submodule of $R$-module $R.$

\begin{proposition}
	\label{mainpro}Prime submodules $\Rightarrow$ 1-absorbing prime submodules
	$\Rightarrow$ 2-absorbing submodules.
\end{proposition}

\begin{proof}
	Let $N$ be a prime submodule of $M.$ Take non-unit elements $a,b\in R$ and  $%
	m\in M$ such that $abm\in N.$ Since $N$ is a prime submodule, $ab\in(N:M)$
	or  $m\in N,$ as desired. Suppose $N$ is a 1-absorbing prime submodule of $M.
	$  Take any $a,b\in R$ and $m\in M$ such that $abm\in N.$ We must obtain
	that  $ab\in(N:M)$ or $am\in N$ or $bm\in N$. If $a,b$ are non-units, we
	have  $ab\in(N:M)$ or $m\in N,$ it is done. Without loss generality, assume $%
	a$ is  unit. Then $abm\in N$ implies $bm\in N,$ as desired.
\end{proof}

\begin{example}
	Let $(R,\mathfrak{m})$ be a local ring with $\mathfrak{m}^{2}=(0_{R})$ and $M$
	be a $R$-module. Then every proper submodule is a 1-absorbing prime submodule
	of $M.$ To see this, choose non-units $a,b\in R$ and $m\in M$ such that
	$abm\in N.$ Since $ab\in\mathfrak{m}^{2}=(0_{R}),$ we have $ab\in(N:M),$ which
	implies $N$ is a 1-absorbing prime submodule of $M.$
\end{example}

\begin{example}
	\label{ex2} \textbf{(1-absorbing prime submodule that is not prime)}Consider $%
	\mathbb{Z}
	_{4}-$module $%
	\mathbb{Z}
	_{4}[X]$ and the submodule $N=<X>$. By previous example, $N$ is a 1-absorbing
	prime submodule. But $N$ is not a prime submodule of $%
	\mathbb{Z}
	_{4}[X].$
\end{example}

\begin{example}
	\label{ex3}	\textbf{(2-absorbing submodule that is not 1-absorbing prime)} Let consider $%
	\mathbb{Z}
	-$module $%
	\mathbb{Z}
	_{30}.$ Suppose that $N$ is the cyclic submodule of $%
	\mathbb{Z}
	-$module $%
	\mathbb{Z}
	_{30}$ generated by $\overline{6},$ that is, $N=<\overline{6}>$. It is clear
	that $<\overline{6}>$ is a 2-absorbing submodule of $%
	\mathbb{Z}
	-$module $%
	\mathbb{Z}
	_{30},$ but it is not a 1-absorbing prime submodule of $%
	\mathbb{Z}
	_{30}.$ Indeed, $2\cdot2\cdot\overline{3}\in<\overline{6}>$ but $\overline
	{4}\notin(N:%
	\mathbb{Z}
	_{30})$ and $3\notin<\overline{6}>.$
\end{example}

\begin{figure}[h]
	\includegraphics[width=1.1\textwidth, ]{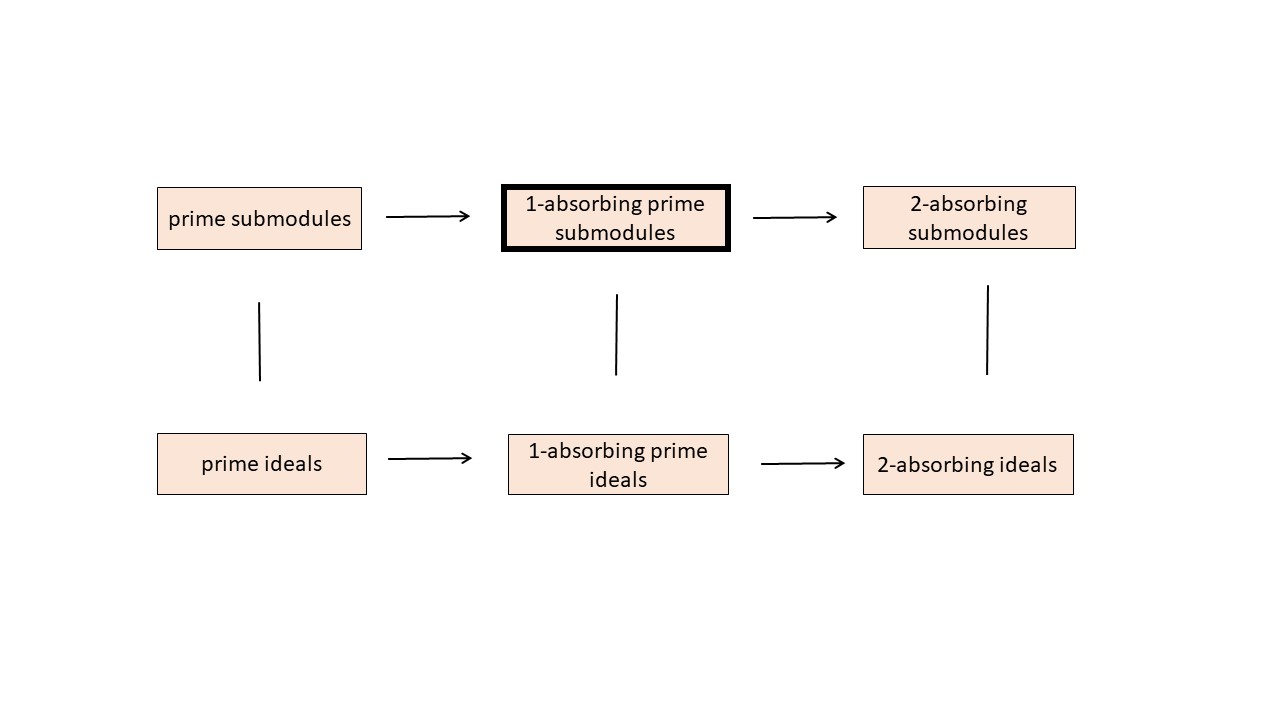}
	\caption{1-absorbing prime submodules (ideals) vs other classical submodules
		(ideals)}
	\label{fig:figure1}
\end{figure}

\begin{proposition}
	\label{pro1} If $N$ is a 1-absorbing prime submodule of $M,$
	
	\begin{enumerate}
		\item $(N:M)$ is a 1-absorbing prime ideal of $R.$
		
		\item $(N:m)$ is a 1-absorbing prime ideal of $R$, for every $m\in M\backslash
		N.$
	\end{enumerate}
\end{proposition}

\begin{proof}
	Let $N$ be a 1-absorbing prime submodule of $M.$
	
	(1): Choose non-units $a,b,c\in R$ such that $abc\in(N:M).$ For all $m\in M$
	then $abcm\in N.$ By our hypothesis, $ab\in(N:M)$ or $cm\in N.$ This implies
	that $ab\in(N:M)$ or $c\in(N:M)$. Consequently, $(N:M)$ is 1-absorbing prime
	ideal of $R.$
	
	(2): Similar to (1).
\end{proof}

The converse of the item (1) of above Proposition is not true for general
cases. For this see the following example:

\begin{example}
	\textbf{ }Let $M=%
	\mathbb{Z}
	\times%
	\mathbb{Z}
	$ and $R=%
	\mathbb{Z}
	.$ Consider $N=<(3,0)>=%
	\mathbb{Z}
	(3,0).$ Then it is clear that $(N:M)=(0).$ Then $(N:M)$ is a prime ideal of $%
	\mathbb{Z}
	$, so it is a 1-absorbing prime ideal of $%
	\mathbb{Z}
	.$ But $%
	\mathbb{Z}
	(3,0)$ is not a 1-absorbing prime submodule of $%
	\mathbb{Z}
	\times%
	\mathbb{Z}
	$. Indeed, take non-units $3,2\in%
	\mathbb{Z}
	$ and $(1,0)\in%
	\mathbb{Z}
	\times%
	\mathbb{Z}
	,$ then $3\cdot2\cdot(1,0)\in%
	\mathbb{Z}
	(3,0).$ But $6\notin(N:M)=(0)$ and $(1,0)\notin%
	\mathbb{Z}
	(3,0).$ Thus $%
	\mathbb{Z}
	(3,0)$ is not a 1-absorbing prime submodule of $%
	\mathbb{Z}
	\times%
	\mathbb{Z}
	$.
\end{example}

\bigskip By the help of the next Lemma, we will prove that $N$ is a
1-absorbing prime submodule of $M\Leftrightarrow$ $IJK\subseteq N$ implies
that either $IJ\subseteq(N:M)$ or $K\subseteq N$ for any two proper ideals $%
I,J$ and a proper submodule $K$ of $M.$

\begin{lemma}
	\label{lemma}Let $M\ $be an $R$-module and $N$ be a 1-absorbing prime
	submodule of $M$. For a proper submodule $K$ of $M$ and for non-unit elements
	$a,b\in R$ if $abK\subseteq N,$ then $ab\in(N:M)$ or $K\subseteq N. $
\end{lemma}

\begin{proof}
	Suppose that $abK\subseteq N$ for non-unit elements $a,b\in R.$ Assume  $%
	K\nsubseteq N.$ Then there is an element $0\neq m\in K\backslash N.$ By 
	assumption, we have $abm\in N.$ Since $N$ is a 1-absorbing prime submodule, 
	either $ab\in(N:M)$ or $m\in N.$ The second one implies a contradiction, we 
	conclude $ab\in(N:M),$ it is done.
\end{proof}

\begin{theorem}
	\label{theoremN}Let $M\ $be an $R$-module and $N$ be a proper submodule of
	$M.$ Then the following statements are equivalent:
	
	\begin{enumerate}
		\item $N$ is a 1-absorbing prime submodule of $M.$
		
		\item If $IJK\subseteq N,$ then either $IJ\subseteq(N:M)$ or $K\subseteq N$
		for any two proper ideals $I,J$ and for a proper submodule $K$ of $M.$
	\end{enumerate}
\end{theorem}

\begin{proof}
	Let $N$ be a proper submodule of $M.$
	
	$(1)\Rightarrow:(2):$ Choose any two proper ideals $I,J$ and a proper 
	submodule $K$ of $M$ such that $IJK\subseteq N.$ Suppose $IJ\nsubseteq(N:M).$
	Then there are non-units $a,b\in R$ such that $a\in I,$ $b\in J$ and $ab\in 
	IJ\backslash(N:M).$ This means that $abK\subseteq N.$ Thus we obtain that  $%
	K\subseteq N$ by Lemma \ref{lemma}.
	
	$(2)\Rightarrow:(1):$ Take non-unit elements $a,b\in R$ and $m\in M$ such
	that  $abm\in N.$ Assume $m\notin N.$ This means that $<m>\nsubseteq N.$
	Consider  $I=(a),J=(b)$ and $K=<m>$. Thus since $IJK\subseteq N$ and $%
	K\nsubseteq N,$ we  conclude $IJ\subseteq(N:M)$ by our hypothesis.
	Consequently, it means that  $ab\in(N:M),$ as desired.
\end{proof}

Note that if $R$ has exactly one maximal ideal, then $R$ is called a \textit{%
	quasilocal ring}.

\begin{theorem}
	\label{theoremQ}Let $M\ $be an $R$-module. If $N$ is a 1-absorbing prime submodule of $M$ that
	is not a prime submodule, then $R$ is a quasilocal ring.
\end{theorem}

\begin{proof}
	Suppose that $N$ is a 1-absorbing prime submodule of $M$ that is not a prime
	submodule. Then there exist a non-unit $r\in R$ and $m\in M$ such that $%
	rm\in N$  but $r\notin(N:M)$ and $m\notin N.$ Choose a non-unit element $%
	s\in R.$ Hence  we have that $rsm\in N$ and $m\notin N.$ Since $N$ is
	1-absorbing prime,  $rs\in(N:M).$ Let us take a unit element $u\in R.$ We
	claim that $s+u$ is a  unit element of $R.$ To see this, assume $s+u$ is
	non-unit. Then $r(s+u)m\in  N$. As $N$ is 1-absorbing prime, $r(s+u)\in(N:M).
	$ This means that  $ru\in(N:M),$ i.e., $r\in(N:M),$ which is a
	contradiction. Thus for any non-unit element $s$ and unit element $u$ in $R,$
	we have $s+u$ is a unit element. Similar to the proof of Theorem 2.4 in \cite%
	{yassine}, we obtain $R$ is a quasilocal ring.
\end{proof}

\begin{corollary}
	\label{corQ}	Let $M\ $be an $R$-module and $R$ is not a quasilocal ring. Then a proper
	submodule $N$ of $M$ is a 1-absorbing prime submodule if and only if $N$ is a
	prime submodule of $M$.
\end{corollary}

\begin{proposition}
	\label{prounion}Let $\{N_{i}\}_{i\in\vartriangle}$ be a chain of 1-absorbing
	prime submodules of $R$-module $M.$ Then the followings hold:
	
	\begin{enumerate}
		\item $\underset{i\in\vartriangle}{\bigcap}N_{i}$ is a 1-absorbing prime
		submodule of $M.$
		
		\item If $M$ is a finitely generated $R$-module, then $\underset
		{i\in\vartriangle}{\bigcup}N_{i}$ is a 1-absorbing prime submodule of $M.$
	\end{enumerate}
\end{proposition}

\begin{proof}
	Let $\{N_{i}\}_{i\in\vartriangle}$ be a chain of 1-absorbing prime
	submodules  of $M.$
	
	(1): Take non-unit elements $a,b\in R$ and $m\in M$ such that $abm\in  
	\underset{i\in\vartriangle}{\bigcap}N_{i}.$ Assume that $m\notin\underset{%
		i\in\vartriangle}{\bigcap}N_{i},$ so there exists $i\in\vartriangle$ such 
	that $m\notin N_{i}.$ Since $N_{i}$ is 1-absorbing prime, we conclude  $%
	ab\in(N_{i}:M).$ For any $j\in\vartriangle,$ we two cases. If $%
	N_{i}\subseteq  N_{j},$ then $(N_{i}:M)\subseteq(N_{j}:M),$ that is, $%
	ab\in(N_{j}:M).$ If  $N_{j}\subset N_{i},$ then we obtain that $%
	ab\in(N_{j}:M)$ since $m\notin  N_{j}$ and $N_{j}$ is 1-absorbing prime. As
	a consequence, for all cases we  have $ab\in(\underset{i\in\vartriangle}{%
		\bigcap}N_{i}:M).$
	
	(2): Since $M$ is finitely generated, $\underset{i\in\vartriangle}{\bigcup }%
	N_{i}$ is a proper submodule of $M.$ Choose non-unit elements $a,b\in R$ and
	$m\in M$ such that $abm\in\underset{i\in\vartriangle}{\bigcup}N_{i}$ and  $%
	m\notin\underset{i\in\vartriangle}{\bigcup}N_{i}.$ Thus for $i\in 
	\vartriangle,$ $abm\in N_{i}$ and $m\notin N_{i}.$ This gives us $ab\in 
	(N_{i}:M)\subseteq(\underset{i\in\vartriangle}{\bigcup}N_{i}:M).$ It is done.
\end{proof}

\begin{proposition}
	Let $f:M\rightarrow M^{\prime}$ be a homomorphism of $R$-module $M$ and
	$M^{\prime}$. Then the followings hold:
	
	\begin{enumerate}
		\item If $N^{\prime}$ is a 1-absorbing
		prime submodule of $M^{\prime}$ with $f^{-1}(N^{\prime})\neq M$, then
		$f^{-1}(N^{\prime})$ is a 1-absorbing prime submodule of $M$.
		
		\item Suppose that $f$ is an epimorphism. If $N$ is a 1-absorbing prime
		submodule of $M$ containing $Ker(f)$, then $f(N)$ is a 1-absorbing prime
		submodule of $M^{\prime}.$
	\end{enumerate}
\end{proposition}

\begin{proof}
	Let $f:M\rightarrow M^{\prime}$ be a homomorphism of $R$-module $M$ and  $%
	M^{\prime}$.
	
	(1): Let choose non-units $a,b\in R$ and $m\in M$ such that $abm\in 
	f^{-1}(N^{\prime}).$ This means that $abf(m)=f(abm)\in N^{\prime}.$ Since  $%
	N^{\prime}$ is a 1-absorbing prime submodule of $M^{\prime},$ we have  $%
	ab\in(N^{\prime}:M^{\prime})$ or $f(m)\in N^{\prime}.$ Then we have $%
	ab\in(f^{-1}(N^{\prime}):M)$ or $m\in f^{-1}(N^{\prime }).$
	
	(2): Take non-units $a,b\in R$ and $m^{\prime}\in M^{\prime}$ such that  $%
	abm^{\prime}\in f(N).$ By assumption there exists $m\in M$ such that  $%
	m^{\prime}=f(m)$ and so $f(abm)\in f(N).$ Then $abm\in f^{-1}(f(N))\subseteq
	N$, as $Ker(f)\subseteq N.$ This implies that either $ab\in(N:M)$ or $m\in N.
	$  If $ab\in(N:M)$, then $abM\subseteq N,$ that is, $abf(M)=abM^{\prime}%
	\subseteq  f(N).$ Thus $ab\in(f(N):M^{\prime})$, it is done. If $m\in N,$
	then  $m^{\prime}=f(m)\in f(N),$ as required.
\end{proof}

By previous Proposition, one can easily obtain the following result:

\begin{corollary}
	\label{corbolum}Let $M$ be an $R$-module and $K\subset N$ be submodules of
	$M$. If $N$ is a 1-absorbing prime submodule of $M$, then $N/K$ is a
	1-absorbing prime submodule of $M/K$.
\end{corollary}

\begin{definition}
	Let $M\ $be an $R$-module and $N$ be a proper submodule of $M.$ Let $P$ be a
	1-absorbing prime submodule of $M$ such that $N\subseteq P.$ If there is not
	exist a 1-absorbing prime submodule $P^{\prime}$ such that $N\subseteq
	P^{\prime}\subset P,$ then $P$ is called a \textbf{minimal 1-absorbing prime
		submodule of }$N.$
\end{definition}

\begin{theorem}
	\label{zorn}Let $M\ $be an $R$-module and $N$ be a proper submodule of $M.$ If
	$P$ is a 1-absorbing prime submodule of $M$ such that $N\subseteq P,$ then
	there exists a minimal 1-absorbing prime submodule of $N$ that it is contained
	in $P.$
\end{theorem}

\begin{proof}
	Let define $\Lambda:=\{P_{i}\in S(M):$ $P_{i}$ is a 1-absorbing prime 
	submodule of $M$ such that $N\subseteq P_{i}\subseteq P\}.$ Since $%
	N\subseteq  P,$ we get $\Lambda\neq\emptyset.$ Consider $(\Lambda,\supseteq).
	$ Let us take a chain $\{N_{i}\}_{i\in\vartriangle}$ in $\Lambda.$ By 
	Proposition \ref{prounion}(1), since $\underset{i\in\vartriangle}{\bigcap }%
	N_{i}$ is a 1-absorbing prime, we can use Zorn's Lemma. Thus there exists a 
	maximal element $K\in\Lambda.$ Then $A$ is 1-absorbing prime and $N\subseteq
	K\subseteq P.$ Now we will show that $K$ is a minimal 1-absorbing prime 
	submodule of $N.$ For the contrary, assume that there exists a 1-absorbing 
	prime submodule $K^{\prime}$ such that $N\subseteq K^{\prime}\subseteq K.$ 
	Then $K^{\prime}\in\Lambda$ and $K\leq K^{\prime}.$ This implies $K= K^{\prime}.$ Consequently, $K$ is a minimal 
	1-absorbing prime submodule of $N.$
\end{proof}

\begin{corollary}
	\label{cor3item}Let $M\ $be an $R$-module and $N$ be a proper submodule of
	$M.$
	
	\begin{enumerate}
		\item Every 1-absorbing prime submodule of $M$ contains at least one minimal
		1-absorbing prime submodule of $M.$
		
		\item If $M$ is finitely generated, then every proper submodule of $M$ has at
		least one minimal 1-absorbing prime submodule of $M.$
		
		\item If $M$ is finitely generated, then there exists a 1-absorbing prime
		submodule of $M$ such that contains $N.$
	\end{enumerate}
\end{corollary}

\begin{proof}
	Let $N$ be a proper submodule of $M.$
	
	(1): Obvious by Theorem \ref{zorn}.
	
	(2): Let $M\ $be a finitely generated $R$-module and $N$ be a proper
	submodule  of $M.$ Then there exists a prime submodule $P$ such that $%
	N\subseteq P,$ see  \cite{MM}. As every prime is 1-absorbing prime, $P$ is a
	1-absorbing prime  submodule. Thus, by the item (1), it is done.
	
	(3): By the item (2), it is clear.
\end{proof}

\begin{definition}
	For any $I\in Id(R),$ we define
	\[
	\Omega:=\{I_{i}\in Id(R):\ I_{i}\text{ is a 1-absorbing prime ideal such that
	}I\subseteq I_{i}\}.
	\]
	Then the intersection of all elements in $\Omega$ is called $radical_{1}%
	$\textbf{\ }of\textbf{\ }$I,$ and we denote it as
	\[
	rad_{1}(I):=\underset{I_{i}\in\Omega}{\bigcap}I_{i}\text{ and if }%
	\Omega=\emptyset\text{ or }I=R,\text{ we define }rad_{1}(I):=R.\text{ }%
	\]
	
\end{definition}

\begin{remark}
	It is clear that $rad_{1}(I)\subseteq\sqrt{I}=rad(I),$ since every prime ideal
	is a 1-absorbing prime ideal.
\end{remark}

\begin{definition}
	For any $N\in S(M),$ we define
	\[
	\Omega:=\{P_{i}\in S(M):\ P_{i}\text{ is a 1-absorbing prime submodule such
		that }N\subseteq P_{i}\}.
	\]
	Then the intersection of all elements in $\Omega$ is called $radical_{1}$ of
	$N,$ and we denote it as
	\[
	rad_{1}(N):=\underset{P_{i}\in\Omega}{\bigcap}P_{i}\text{ and if }%
	\Omega=\emptyset\text{ or }N=M,\text{ we define }rad_{1}(N):=M.\text{ }%
	\]
	
\end{definition}

\begin{remark}
	It is clear that $rad_{1}(N)\subseteq rad(N),$ since every prime submodule is
	a 1-absorbing prime submodule.
\end{remark}

\begin{proposition}
	Let $M\ $be an $R$-module and $N$, $L$ be two submodules of $M.$ The following
	statements are hold:
\end{proposition}

\begin{enumerate}
	\item $N\subseteq rad_{1}(N),$
	
	\item $rad_{1}(rad_{1}(N))\subseteq rad_{1}(N),$
	
	\item $rad_{1}(N\cap L)\subseteq rad_{1}(N)\cap rad_{1}(L),$
	
	\item $rad_{1}(IM)\subseteq rad_{1}(\sqrt{I}M).$
	
	\item $rad_{1}(N:M)\subseteq(rad_{1}(N):M).$
\end{enumerate}

\begin{proof}
	The first four items are elementary.
	
	(5): If $rad_{1}(N)=M,$ it is trivial. Let $rad_{1}(N)\neq M.$ Then there is
	a  1-absorbing prime submodule $K$ of $M$ such that $N\subseteq K.$ Thus $%
	(K:M)$  is a 1-absorbing prime ideal with $(N:M)$ $\subseteq$ $(K:M).$ This
	means that  $rad_{1}(N:M)\subseteq(K:M),$ that is, $rad_{1}(N:M)M\subseteq
	K. $ The  containment is held for all 1-absorbing prime submodule $K_{i}$ of 
	$M$ such  that $N\subseteq K_{i}.$ This implies that $rad_{1}(N:M)M\subseteq
	rad_{1}(N),$ i.e., $rad_{1}(N:M)\subseteq(rad_{1}(N):M).$
\end{proof}

\begin{proposition}
	Let $M\ $be a finitely generated $R$-module and $N$ be a submodule of $M.$
	Then $rad_{1}(N)=M$ if and only if $N=M.$
\end{proposition}

\begin{proof}
	Let $rad_{1}(N)=M.$ Suppose that $N\neq M.$ By Corollary \ref{cor3item}(3), 
	there exists a 1-absorbing prime submodule $N^{\prime}$ of $M$ such that  $%
	N\subseteq N^{\prime}.$ Hence we conclude that $rad_{1}(N)=M\subseteq$  $%
	N^{\prime},$ a contradiction. The other way is clear.
\end{proof}

\begin{theorem}
	Let $M\ $be a finitely generated $R$-module and $N$, $L$ be two submodules of
	$M.$ Then $N+L=M$ $\Leftrightarrow$ $rad_{1}(N)+rad_{1}(L)=M.$
\end{theorem}

\begin{proof}
	Let $M\ $be a finitely generated $R$-module.
	
	$\Rightarrow:$ Let $N+L=M$. We know that $N\subseteq rad_{1}(N)$ and  $%
	L\subseteq rad_{1}(L).$ Thus $M=N+L\subseteq rad_{1}(N)+rad_{1}(L),$ it is
	done.
	
	$\Leftarrow:$ \ Suppose that $N+L\neq M.$ Since $M\ $is finitely generated, 
	there is a maximal submodule $K$ of $M$ such that $N+L\subseteq K,$ see  
	\cite{MM}. Furthermore, $K$ is a prime submodule (so 1-absorbing prime 
	submodule). As $rad_{1}(N)\subseteq K$ and $rad_{1}(L)\subseteq K,$ we 
	conclude $rad_{1}(N)+rad_{1}(L)\subseteq K,$ that is $M\subseteq K.$ This 
	contradicts our hypothesis.
\end{proof}

\section{Properties of 1-Absorbing Prime Submodules of Multiplication Modules}

For the integrity of our study, we would like to give some information about
multiplication module. An $R$-module $M$ is called a \textit{multiplication
	module} if every submodule $N$ of $M$ has the form $IM$ for some ideal $I$
of $R$, see \cite{smith}. Note that, since $I\subseteq(N:M)$ then $%
N=IM\subseteq (N:M)M\subseteq N$. So, if $M$ is multiplication, $N=(N:M)M$,
for every submodule $N$ of $M.$ Afterwards, a multiplication $R$-module $M$
is characterized by a maximal ideal of $R.$ Let $Q$ be a maximal ideal of $R.
$ To give the characterization, let us define the submodule $T_{Q}(M)=\{m\in
M:$ there is a $q\in Q$ such that $(1-q)m=0_{M}\}$ of $M.$ If $T_{Q}(M)=M,$
we call $M$ is a $Q$\textit{-torsion module.} If there exist $q\in Q$ and $%
m\in M$ such that $(1-q)M\subseteq Rm,$ we say $M$ is a $Q$\textit{-cyclic
	module. }Then the authors proved that $M$ is a multiplication $R$-module $%
\Leftrightarrow$ for every maximal ideal $Q$ of $R,$ either $M$ is a $Q$%
-cyclic module or $M$ is a $Q$-torsion module, see Theorem 1.2 in \cite%
{smith}. For more information about multiplication modules, we refer \cite%
{Bar} and \cite{anderson} to the reader.

\begin{theorem}
	\label{main}Let $M$ be a faithful and multiplication $R$-module. Let $I$ be a
	1-absorbing prime ideal of $R$. Then $abm\in IM$ implies $ab\in I$ or $m\in IM$ for
	all non-units $a,b\in R$ and $m\in M.$
\end{theorem}

\begin{proof}
	Choose non-units $a,b\in R$ and $m\in M$ such that $abm\in IM.$ Suppose  $%
	ab\notin I.$ Let us define $I^{\prime}:=\{r\in R:rm\in IM\}.$ If $I^{\prime
	}=R,$ it is done. If $I^{\prime}\neq R,$ there is a maximal ideal $Q$ of  $R$%
	\ such that $I^{\prime}\subseteq Q.$ Now, we claim that $m\notin T_{Q}(M).$ 
	Indeed, if $m\in T_{Q}(M),$ there is an element $q\in Q$ such that  $%
	(1-q)m=0_{M}.$ This means that $1-q\in I^{\prime}\subseteq Q,$ a 
	contradiction. Thus $T_{Q}(M)\neq M.$ Since $M$ is a multiplication module,  
	$M$ is a $Q$-cyclic module, by Theorem 1.2 in \cite{smith}. Hence there is  $%
	q^{\prime}\in Q$ and $m^{\prime}\in M$ such that $(1-q^{\prime})M\subseteq 
	Rm^{\prime}.$ Then $(1-q^{\prime})m\in Rm^{\prime},$ there is $s\in R$ such
	that  $(1-q^{\prime})m=sm^{\prime}.$ Then $(1-q^{\prime})abm=sabm^{\prime}%
	\in IM$  and $(1-q^{\prime})abm\in Rm^{\prime}.$ Thus there are $%
	a^{\prime}\in I$ such  that $(1-q^{\prime})abm=a^{\prime}m^{\prime}.$ Since $%
	sabm^{\prime}=a^{\prime }m^{\prime},$ we obtain $abs-$ $a^{\prime}\in
	Ann(m^{\prime}).$ On the other  hand, $(1-q^{\prime})M\subseteq Rm^{\prime}$
	implies that $(1-q^{\prime })Ann(m^{\prime})M\subseteq
	RAnn(m^{\prime})m^{\prime}=0_{M},$ that is,  $(1-q^{\prime})Ann(m^{\prime})%
	\subseteq Ann(M).$ As $M$ is faithful,  $(1-q^{\prime})Ann(m^{\prime})=0_{R}.
	$ Then $(1-q^{\prime})(abs-$ $a^{\prime })=0_{R}.$ Thus we have $%
	abs(1-q^{\prime})=a^{\prime}(1-q^{\prime})\in I.$  Then $abs(1-q^{\prime})%
	\in I.$ Here we have 2 situations for $s\in R:$
	
	Case 1: Let $s$ be unit. Then we have $ab(1-q^{\prime})\in I.$ If  $%
	1-q^{\prime}$ is a unit element of $R,$ then $ab\in I.$ This contradicts our
	assumption $ab\notin I$. Let $1-q^{\prime}$ be non-unit. Since $I$ is a 
	1-absorbing prime, $ab\in I$ (again, it is not possible) or $1-q^{\prime}\in
	I.$ If $\ 1-q^{\prime}\in I,$ then we have $(1-q^{\prime})m\in IM,$ so  $%
	1-q^{\prime}\in I^{\prime}\subseteq Q,$ a contradiction.
	
	Case 2: Let $s$ be non-unit. Now, we have 2 possibilities for $1-q^{\prime}.$
	If $1-q^{\prime}$ is a unit element of $R,$ then $sab\in I.$ Since $I$ is a 
	1-absorbing prime, $ab\in I$ (again, it is not possible) or $s\in I.$  Then$%
	\ sm^{\prime}\in IM$. Since $sm^{\prime}=(1-q^{\prime})m,$ we have  $%
	(1-q^{\prime})m\in IM$. So $1-q^{\prime}\in I^{\prime}\subseteq Q,$ not 
	possible. If $1-q^{\prime}$ is non-unit, since $I$ is a 1-absorbing prime, 
	either $abs\in I$ or $1-q^{\prime}\in I.$ Again, since it is 1-absorbing 
	prime, we have $ab\in I$ or $s\in I$ or $1-q^{\prime}\in I.$ All 
	possibilities give us a contradiction because of the above explanations.
	
	As a consequence, $I^{\prime}=R,$ that is, $m\in IM.$
\end{proof}

\begin{corollary}
	\label{maincor}Let $M$ be a faithful and multiplication $R$-module. Let $I$ be
	a 1-absorbing prime of $R$. If $IM\neq M,$ then $IM$ is a 1-absorbing prime
	submodule of $M.$
\end{corollary}

\begin{proof}
	Take non-unit elements $a,b\in R$ and $m\in M$ such that $abm\in IM.$
	Suppose  $ab\notin(IM:M).$ Then $ab\notin I.$ By Theorem \ref{main}, it must
	be $m\in  IM.$ It is done.
\end{proof}

\begin{proposition}
	\label{bolum}Let $I,J$ be two ideals of $R$ such that $I\subseteq J.$ If $J$
	is a 1-absorbing prime ideal of $R,$ then $J/I$ is a 1-absorbing prime ideal
	of $R/I.$
\end{proposition}

\begin{proof}
	Let $I$ be a 1-absorbing prime ideal of $R.$ Choose two non-unit elements  $%
	a+I,b+I,c+I$ in $R/I$ such that $abc+I\in J/I.$ This implies that $abc\in J.$
	As $\{r+I:r\in U(R)\}\subseteq U(R/I),$ we obtain $a,b,c$ are non-units.
	Since  $J$ is 1-absorbing prime, either $ab\in J$ or $c\in J.$ This means $%
	ab+I\in  J/I$ or $c+I\in J/I.$
\end{proof}

\begin{definition}
	Let $I$ be an ideal of $R$. If the following equation holds, then we say $R$
	has a \textbf{good unit element property} \textbf{for }$I.$%
	\[
	U(R/I)=\{r+I:r\in U(R)\}
	\]
	
\end{definition}

\begin{remark}
	For the other way of Proposition \ref{bolum}, we need to the good unit element
	property for $I.$ See the Corollary 2.17 of \cite{yassine}: Let $J$ be an
	ideal of $R$ such that $I\subseteq J$ and $U(R/I)=\{r+I:r\in U(R)\}.$ Then $J$
	is a 1-absorbing prime ideal of $R$ $\Leftrightarrow$ $J/I$ is a 1-absorbing
	prime ideal of $R/I.$
\end{remark}

\begin{proposition}
	\label{bolumring} Let $M$ be an $R$-module and $R$ has a good unit element
	property for\textbf{\ }$Ann(M).$ If $N$ is a 1-absorbing prime submodule as
	$R/Ann(M)$-module $M,$ then $N$ is a 1-absorbing prime submodule of $R$-module
	$M.$
\end{proposition}

\begin{proof}
	Take non-unit elements $a,b\in R$ and $m\in M$ such that $abm\in N.$ We must
	show that either $ab\in(N:_{R}M)$ or $m\in N.$ Consider  $%
	(a+Ann(M))(b+Ann(M))m=abm+Ann(M)m\in N.$ Since $R$ has good unit element 
	property for\textbf{\ }$Ann(M),$ we say $a+Ann(M)$ and $b+Ann(M)$ are
	non-unit  elements of $R/Ann(M).$ As $N$ is a 1-absorbing prime submodule as
	$R/Ann(M)$-module, we obtain that either $ab+Ann(M)\in(N:_{R/Ann(M)}M)$ or  $%
	m\in N.$ If the second one holds, it is done. The first one implies that  $%
	abM\subseteq N,$ that is, $ab\in(N:_{R}M),$ as required.
\end{proof}

\begin{theorem}
	\label{char}Let $M$ be a multiplication $R$-module and $R$ has a good unit
	element property for\textbf{\ }$Ann(M).$ Then the followings are equivalent:
	
	\begin{enumerate}
		\item $N$ is a 1-absorbing prime submodule of $M.$
		
		\item $(N:M)$ is a 1-absorbing prime submodule of $R.$
		
		\item For a proper ideal $I$ of $R$ such that $Ann(M)\subseteq I$ and
		$I$ is a 1-absorbing prime ideal, then $N=IM.$
	\end{enumerate}
\end{theorem}

\begin{proof}
	Let $M$ be a multiplication $R$-module.
	
	$(1)\Rightarrow(2):$ By Proposition \ref{pro1}.
	
	$(2)\Rightarrow(3):$ Consider $I=(N:M).$
	
	$(3)\Rightarrow(1):$ Since $M$ is a multiplication $R$-module, $M$ is a 
	faithful multiplication $R/Ann(M)$-module by page 759 of \cite{smith}. Also,
	as $I$ is a 1-absorbing prime ideal of $R$, $I/Ann(M)$ is a 1-absorbing
	prime  ideal of $R/Ann(M),$ see Proposition \ref{bolum}. Then $[I/Ann(M)]M$
	is a  1-absorbing prime submodule of $R/Ann(M)$-module $M,$ by Corollary  %
	\ref{maincor}. This means that $[I/Ann(M)]M$ is a 1-absorbing prime
	submodule  of $R$-module $M,$ by Proposition \ref{bolumring}. As $%
	[I/Ann(M)]M=N,$  consequently, $N$ is a 1-absorbing prime submodule of $R$%
	-module $M.$
\end{proof}


\begin{thebibliography}{99}                                                                                     \bibitem{RA2003}\setlength{\baselineskip}{.45cm} \textit{Ameri R.,} On the prime submodules of
	multiplication modules, Inter. J. of Math. and Mathematical Sciences, 2003, 
	\textbf{27}, 1715--1724.
	
	\bibitem{anderson}\setlength{\baselineskip}{.45cm} \textit{Anderson D. D., Arabaci T., Tekir U., Koc S., }On
	S-multiplication modules, Communacations in algebra, 2020, 3398-3407.
	
	\bibitem{atiyah}\setlength{\baselineskip}{.45cm} \textit{Atiyah M. F. and MacDonald I. G.,} Introduction to
	commutative algebra, CRC Press, 1969.
	
	\bibitem{Ba}\setlength{\baselineskip}{.45cm}  \textit{Badawi A.,} On 2-absorbing ideals of commutative rings,
	Bull. Austral. Math. Soc., \textbf{75}, 2007, 417-429.
	
	\bibitem{Bar}\setlength{\baselineskip}{.45cm}  \textit{Barnard A.,} Multiplication modules, Journal of
	Algebra, \textbf{71} ,1981, 174-178.
	
	\bibitem{darani}\setlength{\baselineskip}{.45cm}  \textit{Darani A. Y. and Soheilnia F.,} On 2-absorbing and
	weakly 2-absorbing submodules, Thai J. Math., 2011, \textbf{9}, 577-584.
	
	\bibitem{JD1978}\setlength{\baselineskip}{.45cm}  \textit{Dauns J.}, Prime modules, J. reine Angew. Math.,
	1978,\textbf{\ 2,} 156--181.
	
	\bibitem{smith}\setlength{\baselineskip}{.45cm}  \textit{El-Bast Z. A. and Smith P. F.,} Multiplication
	modules, Comm. in Algebra, 1988, \textbf{16,} 755-779.
	
	\bibitem{L}\setlength{\baselineskip}{.45cm}  \textit{Lu C. P.,} Prime submodules of modules, Comm. Math.
	Univ. Sancti Pauli, 1984, \textbf{33}, 61--69.
	
	\bibitem{MM}\setlength{\baselineskip}{.45cm}  \textit{McCasland R. L. and Moore M. E.,} Prime submodules,
	Comm. Algebra, 1992, \textbf{20, }1803--1817.
	
	\bibitem{PF1988}\setlength{\baselineskip}{.45cm}  \textit{Smith P. F.} Some remarks on Multiplication
	modules, Arch Math., 1988, \textbf{50}, 223-235.
	
	\bibitem{yassine}\setlength{\baselineskip}{.45cm}  \textit{Yassine A., Nikmehr M. J. and Nikandish R.,} On
	1-absorbing prime ideals of commutative rings, Journal of Algebra and Its
	Applications, 2020, (to appear).
\end{thebibliography}
\end{document}